\documentclass{article}
\title{The Categorification of a Symmetric Operad is Independent of Signature}
\author{Miles Gould\\University of Glasgow}
\date{November 28, 2007}
\usepackage{amsmath}
\usepackage{amsthm}
\DeclareSymbolFont{AMSb}{U}{msb}{m}{n}
\DeclareMathSymbol{\natural}{\mathbin}{AMSb}{"4E}
\DeclareMathSymbol{\integer}{\mathbin}{AMSb}{"5A}
\DeclareMathSymbol{\real}{\mathbin}{AMSb}{"52}
\DeclareMathSymbol{\rational}{\mathbin}{AMSb}{"51}
\DeclareMathSymbol{\I}{\mathbin}{AMSb}{"49}
\DeclareMathSymbol{\complex}{\mathbin}{AMSb}{"43}
\DeclareMathSymbol{\bbF}{\mathbin}{AMSb}{"46}
\DeclareMathSymbol{\bbI}{\mathbin}{AMSb}{"49}
\newcommand{\C}{\ensuremath{\mathcal{C}}}

\newcommand{\defterm}[1]{\textbf{#1}}
\newcommand{\Alg}{\textbf{Alg}}
\newcommand{\Algwk}{\textbf{Alg}_{\textrm{wk}}}
\newcommand{\Cat}{\textbf{Cat}}
\newcommand{\CAT}{\textbf{CAT}}
\newcommand{\Set}{\textbf{Set}}
\newcommand{\Digraph}{\textbf{Digraph}}
\newcommand{\Operad}{\textbf{Operad}}
\newcommand{\SymmOperad}{\textbf{$\Sigma$-Operad}}
\newcommand{\CatOperad}{\textbf{Cat-Operad}}
\newcommand{\CatSymmOperad}{\textbf{Cat-$\Sigma$-Operad}}
\newcommand{\VOperad}{\textbf{$\mathcal{V}$-Operad}}
\newcommand{\VSymmOperad}{\textbf{$\mathcal{V}$-$\Sigma$-Operad}}

\newcommand{\End}{\mathop{\mbox{\rm{End}}}}
\theoremstyle{plain}
\newtheorem{theorem}{Theorem}[section]
\newtheorem{lemma}[theorem]{Lemma}
\newtheorem{corollary}[theorem]{Corollary}
\theoremstyle{definition}
\newtheorem{defn}[theorem]{Definition}
\newtheorem{example}[theorem]{Example}
\allowdisplaybreaks
\usepackage[all,2cell]{xy}
\UseTwocells

\newcommand{\seq}{_\bullet}
\newcommand{\udot}{^\bullet}
\newcommand{\dseq}{\seq\udot}

\newcommand{\WkP}{{\ensuremath{\mbox{Wk($P$)}}}}

\newcommand{\prodkn}[1]{#1_{k_1} \times \dots \times #1_{k_n}}
\newcommand{\midlabel}[1]{
	        \xymatrixrowsep{4pc}
		\xymatrixcolsep{4pc}
		\xymatrix{ {} \ar @{}[d]^{#1} \\ {} }
}
\newcommand{\midequals}{\midlabel{=}}

\newcommand{\fork}[6]{
	\xymatrix{#1 \ar@<0.5ex>[r]^-{#2} \ar@<-0.5ex>[r]_-{#3} & #4
		\ar[r]^-{#5} & #6
	}
}
\newcommand{\square}[8]{
	\xymatrix{
		#1 \ar[r]^-{#2} \ar[d]_-{#4}
		& #3 \ar[d]^-{#5} \\
		#6 \ar[r]^-{#7}
		& #8
	}
}

\newdir{(>}{{}*!/-5pt/\dir{>}}
\newcommand{\lffar}{\ar@{(>->}} 
\newcommand{\booar}{\ar@{->>}} 
\newcommand{\unar}{\ar@{..>}}  
\newcommand{\trans}{\overline}
\newcommand{\E}{\ensuremath{\mathcal{E}}}
\newcommand{\M}{\ensuremath{\mathcal{M}}}
\newcommand{\Ebar}{\overline \E}
\newcommand{\Mbar}{\overline \M}
\newcommand{\arin}{\mathop{\mbox{in}}}
\newcommand{\orth}{\mathop{\bot}}
\newcommand{\orthset}{^{\bot}}
\newcommand{\smc}{symmetric monoidal category}

\def\nxseq#1{{#1}_n \times {#1}\seq}
\def\sumki{_{\sum{k_i}}}
\def\hatmap{\hat{\phantom{\alpha}}} 

\begin{document}
\maketitle
\begin{abstract}
Given a symmetric operad $P$, and a signature (or generating sequence) $\Phi$
for $P$, we define a notion of the \emph{categorification} (or
\emph{weakening}) of $P$ with respect to $\Phi$.
When $P$ is the symmetric operad whose algebras are commutative monoids, with
the standard signature, we recover the notion of symmetric monoidal
categories.
We then show that this categorification is independent (up to equivalence) of
the choice of signature.
\end{abstract}
\begin{section}{Introduction}
In \cite{hohc}, Leinster showed how to form a categorified version of the
theory of monoids, starting from any signature for that theory, and showed
that for all such signatures, the categorified version was equivalent to
the classical theory of monoidal categories.
One might ask how far this result generalizes: in the present paper, we show
that it can be extended to any theory whose models are algebras for a symmetric
operad.
It is believed that these theories are the ``linear'' ones, i.e. those which
can be presented by means of equations whose variables appear exactly once on
each side, though possibly not in the same order.

Section \ref{sec:background} covers some background material: readers familiar
with the theories of factorization systems and operads can skip most of this,
with the possible exception of Lemmas \ref{lem:fork} and \ref{lem:monad}.
Section \ref{sec:main} covers the definition of the categorification of a linear
theory, and proves that this categorification is independent of the choice of
signature.
Section \ref{sec:symmmoncats} uses this definition to explicitly calculate the
categorification of the theory of commutative monoids with their standard
signature, and shows that this is exactly the classical theory of symmetric
monoidal categories.
Section \ref{sec:psalg} discusses how our result relates to the 2-categorical
notion of pseudo-algebras for a 2-monad, and section \ref{sec:nonlinear}
discusses the difficulties involved in extending our approach to general
finitary theories.

An earlier version of this paper was presented at the 85th Peripatetic Seminar
on Sheaves and Logic in Nice in March 2007.
This version differs mainly in that the main theorem has been expanded to
include symmetric (rather than non-symmetric) operads; in other words, it has
been extended from strongly regular theories to linear theories.
I would like to thank Michael Batanin for suggesting that I work on this
generalization.
The background material is also covered in more detail.
\end{section}
\begin{section}{Background}
\label{sec:background}
We start by recalling some basic notions of operad theory.
For more on operads, see for instance \cite{hohc} Chapter 2.
We borrow the notation $f\seq = (f_1, f_2, \dots, f_n)$ and  $g\dseq = (g_1^1, \dots, g_1^{k_1}, \dots, g_n^1, \dots, g_n^{k_n})$ from chain complexes.
We take the set of natural numbers $\natural$ to include 0.
\begin{defn}
A \defterm{plain operad} $P$ is
\begin{itemize}
\item A sequence of sets $P_0, P_1, P_2, \dots$
\item For all $n, k_1, \dots, k_n \in \natural$, a function $\circ: P_n
\times \prodkn{P} \to P_{\sum k_i}$ 
\item A \defterm{unit element} $1 \in P_1$
\end{itemize}
satisfying the following axioms:
\begin{itemize}
\item \emph{Associativity:} $f \circ (g\seq \circ h\dseq) = (f \circ g\seq)
\circ h\dseq$ wherever this makes sense
\item \emph{Units:} $1 \circ f = f = f \circ (1, \dots, 1)$ for all $f$.
\end{itemize}
\end{defn}

\begin{example}
\label{ex:plainend}
Let $\C$ be a monoidal category and $A$ be an object of $\C$.
Then there is a plain operad $\End(A)$, called the \defterm{endomorphism operad
of $A$} for which $\End(A)_n = \C(A^{\otimes n}, A)$.
Composition in $\End(A)$ is given by composition and tensoring in $\C$.
\end{example}

\begin{defn}
Let $P$ and $Q$ be plain operads.
A \defterm{morphism of plain operads} $f:P \to Q$ is a sequence of functions
$f_n: P_n \to Q_n$ commuting with the composition functions in $P$ and $Q$:
\[
\xymatrixcolsep{6pc}
\xymatrix{
	P_n \times \prod P_{k_i}
		\ar[r]^{f_n \times f_{k_1} \times \dots \times f_{k_n}}
		\ar[d]_{\circ}
	& Q_n \times \prod Q_{k_i} \ar[d]^{\circ} \\
	P_{\sum k_i} \ar[r]^{f_{\sum k_i}}
	& Q_{\sum k_i}
}
\]
\[
f_1(1) = 1
\]
If $X$ is some property of functions (invertibility, say), we say that an
operad morphism $f$ is \defterm{levelwise $X$} if every $f_n$ is $X$.
\end{defn}

\begin{defn}
\xymatrixrowsep{4pc}
\xymatrixcolsep{4pc}
A \defterm{symmetric operad} is a plain operad $P$ together with a left action
of the symmetric group $S_n$ on each $P_n$, which is compatible with the
operadic composition:
\[
\xymatrix{
	P_n \times \prod P_{k_i} \ar[r]^{\sigma \times 1 \times \dots \times 1}
		\ar[d]_{\circ}
	& P_n \times \prod P_{k_i} \ar[d]^{\circ} \\
	P_{\sum k_i} \ar[r]^{\sigma \circ (1, \dots, 1)}
	& P_{\sum k_i}
}
\xymatrix{
}
\]
\end{defn}

\begin{example}
If the monoidal category $\C$ in Example \ref{ex:plainend} is symmetric, then $\End(A)$ acquires the structure of a symmetric operad.
The symmetric groups act by composition with the symmetry map in $\C$.
\end{example}

\begin{defn}
Let $P$ and $Q$ be symmetric operads.
A \defterm{morphism of symmetric operads} $f: P \to Q$ is a morphism of plain
operads which commutes with the actions of the symmetric groups, in the sense
that the diagram
\[
\xymatrix{
P_n \ar[r]^{\sigma \cdot -} \ar[d]_{f_n} & P_n \ar[d]^{f_n} \\
Q_n \ar[r]^{\sigma \cdot -} & Q_n
}
\]
commutes for all $n \in \natural$ and all $\sigma \in S_n$.
\end{defn}
\begin{defn}
Let $P$ be an operad (plain or symmetric).
An \defterm{algebra} for $P$ is a set $A$ and a morphism of (the appropriate
kind of) operads $(\hatmap) : P \to \End(A)$.
So if $p \in P_n$, then $\hat p$ is a function $A^n \to A$.
\end{defn}
This amounts to a function $h_n : P_n \times A^n \to A$ for each $n \in
\natural$, satisfying some obvious axioms.

\begin{defn}
Let $A, B$ be algebras for $P$.
A \defterm{morphism of algebras} $A \to B$ is a function $f: A \to B$ such that
\[
\xymatrix{
P_n \times A^n \ar[r]^{1 \times f^n} \ar[d]_{h_n}
& P_n \times B^n \ar[d]^{h_n} \\
A \ar[r]^f & B
}
\]
commutes for all $n \in \natural$.
\end{defn}

We form categories \Operad\ and $\SymmOperad$ of plain and symmetric operads
and their morphisms.
Given an operad $P$, we form a category $\Alg(P)$ of $P$-algebras and
$P$-algebra morphisms.

The notions of (symmetric) operads and their morphisms can be interpreted in any
closed symmetric monoidal category $\mathcal{V}$ in the obvious way: we call
the resulting category \VOperad\ or \VSymmOperad\ as appropriate.
The algebras for a $\mathcal{V}$-operad are objects of $\mathcal{V}$.
We are particularly interested in the case $\mathcal{V} = \Cat$, with the
monoidal structure given by finite products.
\begin{defn}
Let $f,g: P \to Q$ be morphisms of plain \Cat-operads. 
A \defterm{transformation} $\alpha: f \to g$ is a sequence $(\alpha_n : f_n \to g_n)$ of natural transformations such that 
\begin{eqnarray}
\label{eqn:transcomp}
\xymatrixrowsep{4pc}
\xymatrixcolsep{4pc}
\xymatrix{
	\nxseq P \rtwocell^{\nxseq{f}}_{\nxseq{g}}{*{!(-3,0)\object{\nxseq{\alpha}}}}
		\ar[d]_\circ
	& \nxseq Q \ar[d]^\circ \\
	P\sumki \ar[r]_{g\sumki} & Q\sumki
}
&\midequals
&\xymatrixrowsep{4pc}
\xymatrixcolsep{4pc}
\xymatrix{
	\nxseq P \ar[r]^{\nxseq{f}}
		\ar[d]_\circ
	& \nxseq Q \ar[d]^\circ \\
	P\sumki \rtwocell^{f\sumki}_{g\sumki}{*{!(-2.5,0)\object{\alpha\sumki}}} & Q\sumki
} \\
\label{eqn:transunit}
(\alpha_1)_{1} &= &1_1,
\end{eqnarray}
for all $n, k_1 \dots k_n \in \natural$.

\end{defn}
\begin{defn}
Let $f,g: P \to Q$ be morphisms of symmetric \Cat-operads. 
A \defterm{transformation} $\alpha: f \to g$ is a transformation of morphisms
of plain operads such that
\begin{eqnarray}
\xymatrixrowsep{4pc}
\xymatrixcolsep{4pc}
\label{eqn:transsymm}
\xymatrix{
	P_n \rtwocell^{f_n}_{g_n}{*{!(-1,0)\object{\alpha_n}}}
		\ar[d]_{\sigma\cdot-}
	& Q_n \ar[d]^{\sigma\cdot-} \\
	P_n \ar[r]_{g_n} & Q_n
}
&\midequals
&\xymatrixrowsep{4pc}
\xymatrixcolsep{4pc}
\xymatrix{
	P_n \ar[r]^{f_n}
		\ar[d]_{\sigma\cdot-}
	& Q_n \ar[d]^{\sigma\cdot-} \\
	P_n \rtwocell^{f_n}_{g_n}{*{!(-1,0)\object{\alpha_n}}} & Q_n
}
\end{eqnarray}
for every $n \in \natural$ and every $\sigma \in S_n$.
\end{defn}
We note in passing that these notions are unrelated to the operad
transformations that arise from considering operads as one-object
multicategories.
We refer to the 2-category of plain \Cat-operads, their morphisms and
transformations as \CatOperad, and to the 2-category of symmetric
\Cat-operads, their morphisms and transformations as \CatSymmOperad.

\begin{defn}
Let $A$, $B$ be algebras for some symmetric \Cat-operad $P$.
A \defterm{weak morphism of $P$-algebras} $(F, \xi): A \to B$ is a functor $F:
A \to B$ and a sequence $(\xi_n)$ of natural transformations
\[
\xymatrix{
\xymatrixrowsep{4pc}
P_n \times A^n \ar[r]^{1 \times F^n} \ar[d]_{h_n}
	\drtwocell\omit{*{!(-1,-1.5)\object{\xi_n}}}
& P_n \times B^n \ar[d]^{h'_n} \\
A \ar[r]^F & B
}
\] 
satisfying the equations given in \cite{hohc} Section 3.2 Fig. 3-A, and in
addition some extra diagrams expressing the compatibility with the symmetric
group actions.
We call the category of $P$-algebras and weak morphisms $\Algwk(P)$.
\end{defn}

Recall that a \defterm{fork} is a diagram \[\fork A f g B h C\] such that $hf =
hg$, but for which $h$ is not necessarily a coequalizer for $f, g$.

\begin{lemma}
\label{lem:fork}
In \CatSymmOperad, if $\fork P \alpha \beta Q \gamma R$ is a fork, and $\gamma$
is levelwise full and faithful, then $\alpha \cong \beta$.
\end{lemma}
\begin{proof}
We shall construct an invertible \Cat-$\Sigma$-operad transformation
$\eta:\alpha \to \beta$.
We form the $\eta_n$s as follows.
For all $p \in P_n$, $\gamma\alpha(p) = \gamma\beta(p)$.
Since $\gamma$ is levelwise full, there exists an arrow $(\eta_n)_p: \alpha(p)
\to \beta(p)$ such that $\gamma_n(p) = 1_{\gamma\alpha(p)}$.
Since $\gamma$ is levelwise full and faithful, this arrow is an isomorphism.
Each $\eta_n$ is easily seen to be natural.
It remains to show that the collection $(\eta_n)_{n \in \natural}$ forms a
\Cat-$\Sigma$-operad transformation, in other words that the equations
(\ref{eqn:transcomp}), (\ref{eqn:transunit}), and (\ref{eqn:transsymm}) hold.
Since $\gamma$ is levelwise full and faithful, it is enough to show that the
images of both sides under $\gamma$ are equal, and this is trivially true by
definition of $\eta$.
\end{proof}

\begin{defn}
Let $e: a \to b, m: c \to d$ be arrows in a category \C.
We say that $e$ is \defterm{left orthogonal} to $m$, written $e \orth m$, if,
for all arrows $f: a \to c$ and $g: b \to d$, there exists a unique map $t: b
\to c$ such that the following diagram commutes:
\[
\xymatrix{
	a \ar[r]^{\forall f} \ar[d]_e
	& c \ar[d]^m \\
	b \unar[ur]^{\exists! t} \ar[r]_{\forall g}
	& d
}
\]
\end{defn}

\begin{defn} \label{def:FS}
Let \C\ be a category.
A \defterm{factorization system} on \C\ is a pair $(\E, \M)$ of classes of maps
in \C\ such that
\begin{enumerate}
\item \label{axm:factor} for all maps $f$ in \C, there exist $e \in \E$ and $m
\in \M$ such that $f = m \circ e$.
\item \label{axm:iso_closure} \E\ and \M\ contain all isomorphisms, and are
closed under composition with isomorphisms on both sides.
\item \label{axm:orthogonal} $\E \orth \M$, i.e. $e \orth m$ for all $e \in
\E$ and $m \in \M$.
\end{enumerate}
\end{defn}
\begin{example}
\label{ex:epi-mono}
Let $\C = \Set$, $\E$ be the epimorphisms, and $\M$ be the monomorphisms.
Then $(\E,\M)$ is a factorization system.
\end{example}
\begin{example} More generally, let $\C$ be some variety of algebras,
$\E$ be the regular epimorphisms (i.e., the surjections), and $\M$ be the
monomorphisms.
Then $(\E,\M)$ is a factorization system.
\end{example}
Let \Digraph\ be the category of directed graphs and graph maps.
\begin{example}
\label{ex:digraphFS}
Let $\C = \Digraph$, $\E$ be the maps bijective on objects, and \M\ be the full
and faithful maps.
Then $(\E,\M)$ is a factorization system.
\end{example}
In deference to Example \ref{ex:epi-mono}, we shall use arrows like
$\xymatrix{{} \booar[r] & {}}$ to denote members of \E\ in commutative
diagrams, and arrows like $\xymatrix{{} \lffar[r] & {}}$ to denote members of
$\M$, for whatever values of \E\ and \M\ happen to be in force at the time.

We will use without proof the following standard properties of factorization
systems:
\begin{lemma}
\label{lem:stdfact}
Let \C\ be a category, and $(\E, \M)$ be a factorization system on \C.
\begin{enumerate}
\item $\E \cap \M$ is the class of isomorphisms in \C.
\item \label{lem:factuniq} The factorization in (\ref{axm:factor}) is unique up
to unique isomorphism.
\item The factorization in (\ref{axm:factor}) is functorial; in other words, if
the square
\[
\square A f B g h C {f'} D
\]
commutes, and $f = me, f' = m'e'$, then there is a unique morphism $i$ making
\[
\xymatrix{
	A \ar[r]^{e} \ar[d]_g
	& {} \ar[r]^{m} \unar[d]_i
	& B \ar[d]^h \\
	C \ar[r]^{e'}
	& {} \ar[r]^{m'}
	& D
}
\]
commute.
Thus, \E\ and \M\ can be regarded as functors $\mbox{Ar}(\C) \to \mbox{Ar}(\C)$.
\item \E\ and \M\ are closed under composition.
\item $\E\orthset = \M$ and ${}\orthset M = \E$, where $\E\orthset = \{f
\mbox{ in } \C : e \orth f \mbox{ for all } e \in \E\}$ and $\orthset\M = \{f
\mbox{ in } \C : f \orth m \mbox{ for all } m \in \E\}$.
\end{enumerate}
\end{lemma}
We will also use the following fact:
\begin{lemma}
\label{lem:monad}
Let \C\ be a category with a factorization system $(\E, \M)$.
Let $T$ be a monad on \C\ and let $\Ebar = \{ f \arin\ \C : U f \in \E\}$,
$\Mbar = \{f \arin\ \C : U f \in \M\}$ where $U$ is the forgetful functor
$\C^T \to \C$.
Suppose that $T$ preserves \E-arrows.
Then $(\Ebar, \Mbar)$ is a factorization system on $\C^T$.
\end{lemma}
\begin{proof}
Standard.
See \cite{acc}, Proposition 20.24.
\end{proof}
\begin{example} 
Let $(\E, \M)$ be the factorization system on \Digraph\ described in Example
\ref{ex:digraphFS} above, and let $T$ be the free category monad.
Since \Cat\ is monadic over \Digraph, this gives a factorization system $(\Ebar,
\Mbar)$ on \Cat\ where $\Ebar$ is the bijective-on-objects functors, and
$\Mbar$ is the full and faithful ones.
\end{example}
\begin{example} 
\label{ex:catopdFS}
Let $\C = \Cat^\natural$, \E\ be the bijective-on-objects maps, and \M\ be those
that are levelwise full and faithful.
Since \CatSymmOperad\ is monadic over $\Cat^\natural$, this gives a
factorization system $(\Ebar, \Mbar)$ on \CatSymmOperad\ where $\Ebar$ is the
bijective-on-objects
maps, and $\Mbar$ is the levelwise full and faithful ones.
\end{example}

We shall need one final piece of background:
{
\def\SetX{\ensuremath{\Set^X}}
\def\SetXT{\ensuremath{(\SetX)^T}}
\begin{theorem}
If $X$ is a set and $T$ is a monad on \SetX\ then the regular epis in
\SetXT\ are the coordinate-wise surjections.
In other words, the forgetful functor $U:\SetXT \to \SetX$ preserves and
reflects regular epis.
\end{theorem}
\begin{proof}
See \cite{acc} section 20, in particular Definition 20.21 and Proposition 20.30.
\end{proof}
}

\end{section}
\begin{section}{Categorification}
\label{sec:main}
Throughout, let $P$ be a symmetric \Set-operad.
Symmetric operads are algebras for a straightforward multi-sorted algebraic
theory, so (by standard arguments from universal algebra) there is a monadic
adjunction
\[
\xymatrix{\Set^\natural \ar@<1.2ex>[r]^-F & \SymmOperad \ar@<1.2ex>[l]^-U_-\bot}
\]
\begin{defn}
A \defterm{signature} for $P$ is a pair $(\Phi, \phi)$, where $\Phi \in
\Set^\natural$ and $\phi :F\Phi \to P$ is a regular epi.
Where $\phi$ is obvious, we shall abuse notation and refer to $\Phi$ as a
signature.
\end{defn}
\begin{example}
\label{ex:unbiasedsig}
$(UP, \epsilon)$ is a signature for $P$, where $\epsilon$ is the counit of the
adjunction $F \dashv U$, since
\[
\fork{FUFUP}{\epsilon FU}{FU \epsilon}{FUP} {\epsilon}{P}
\]
is a coequalizer diagram.
\end{example}
\begin{example}
\label{ex:commsig}
The sequence $(\{e\}, \emptyset, \{.\}, \emptyset, \emptyset, \dots)$ is a
signature for the terminal symmetric operad, whose algebras are commutative
monoids.
\end{example}

We will define a categorification of any symmetric operad $P$, dependent on a
signature $(\Phi, \phi)$.
This will be a ``weak'' categorification, in the sense that derived operations
which are equal in $P$ will only be isomorphic in the categorified theory.
We will then show that this is independent of our choice of signature, in
the sense that the symmetric \Cat-operads which arise are equivalent (and thus
have equivalent categories of algebras).
\begin{defn}
Let $(\Phi, \phi)$ be a signature for $P$.

Embed \SymmOperad\ into \CatSymmOperad\ via the (full and faithful) discrete
category functor.
Using the factorization system of Example \ref{ex:catopdFS}, factor $\phi$ as
follows in \CatSymmOperad:
\[
\xymatrix{
	F\Phi \ar[rr]^\phi \booar[dr]_b
	&& P \\
	& Q \lffar[ur]_f
}
\]
where $f$ is full and faithful levelwise, and $b$ is bijective on objects.
Then the \defterm{categorification of $P$ with respect to $(\Phi, \phi)$} is
$Q$.
The uniqueness of $Q$ follows from property (\ref{lem:factuniq}) in Lemma
\ref{lem:stdfact}.
\end{defn}
\begin{example}
\label{ex:smc}
(This example will be explored in greater detail in Section
\ref{sec:symmmoncats}).
Let $P$ be the terminal symmetric operad, whose algebras are commutative
monoids.
Let $\Phi$ be the standard signature for commutative monoids,
i.e. a binary operation and a constant.
Then the categorification $Q$ of $P$ with respect to $\Phi$ is the symmetric
\Cat-operad whose algebras in \Cat\ are symmetric monoidal categories.

Objects of $Q$ are permuted trees of binary and nullary nodes, and there
is a unique arrow between two trees $\tau_1$ and $\tau_2$ if and only if they
evaluate to the same operation in the theory of commutative monoids.
By uniqueness, all diagrams in $Q$ commute.

A $Q$-algebra, therefore, is a category $C$ equipped with a binary functor
$\otimes: \C \times \C \to \C$; an object $I \in \C$; and a natural
transformation from a permuted composite of these functors to another if and
only if they have the same number of arguments. Here a ``permuted composite''
is something like the functor $(A, B, C) \mapsto (A \otimes ((I \otimes C)
\otimes B))$, where variables may be permuted but not repeated or dropped.
Diagrams of these natural transformations will commute if all the morphisms
involved have each variable appearing once on each side; but this is precisely
the coherence theorem for symmetric monoidal categories given in
\cite{catwork} Chapter XI, Theorem 1.1.
In particular, all these natural transformations will be invertible.
\end{example}
\begin{defn}
\label{def:wkp}
The \defterm{unbiased categorification} of $P$ is the categorification arising
from the signature $(UP, \epsilon)$ described in Example \ref{ex:unbiasedsig}.
Call this symmetric \Cat-operad \WkP.
\end{defn}

\begin{defn}
An \defterm{unbiased weak $P$-category} is an algebra for \WkP.
A \defterm{weak $P$-functor} is a weak morphism of \WkP-algebras between two
unbiased weak $P$-categories.
\end{defn}

Recall that \CatSymmOperad\ is a 2-category, so we may talk of two symmetric
\Cat-operads being equivalent.

\begin{theorem}
Every categorification of $P$ is equivalent as a symmetric \Cat-operad to {\rm \WkP}.
\end{theorem}
\begin{proof}
Let $Q$ be the categorification of $P$ with respect to a signature  $(\Phi,
\phi)$.
By the triangle identities, we have a commutative square
\[
\xymatrix{
	F\Phi \ar[r]^\phi \ar[d]_{F\trans\phi}
	& P \ar[d]^1 \\
	FUP \ar[r]^\epsilon
	& P
}
\]
By functoriality of the factorization system, this gives rise to a unique map
$\chi: Q \to \WkP$ such that 
\[
\xymatrix{
	F\Phi \booar[r] \ar@/u0.7cm/[rr]^\phi \ar[d]_{F\trans\phi}
	& Q \lffar[r] \unar[d]^\chi
	& P \ar[d]^1 \\
	FUP \booar[r] \ar@/d0.7cm/[rr]_\epsilon
	& \WkP \lffar[r]
	& P
}
\]
commutes.
We wish to find a pseudo-inverse to $\chi$.

Since \SymmOperad\ is monadic over $\Set^\natural$, a regular epi in
\SymmOperad\ is a pointwise surjection
(intuitively, the fact that $\Phi$ is a signature for $P$ means that
$\Phi$ generates $P$, so $\phi_n : (F\Phi)_n \to P_n$ is surjective).
So we may choose a section $\psi_n$ of $\phi_n : (F\Phi)_n \to P_n$ for all $n
\in \natural$.
So we have a morphism $\psi : UP \to UF\Phi$ in $\Set^\natural$ that is a section of $U\phi$.
We wish to show that
\[
\xymatrix{
	FUP \ar[r]^{\epsilon_P} \ar[d]_{\trans\psi}
	& P \ar[d]^1 \\
	F\Phi \ar[r]^\phi
	& P
}
\]
commutes.
Indeed,
\[
\xymatrixrowsep{0.15pc}
\xymatrix{
& FUP \ar[r]^{\bar \psi} & F\Phi \ar[r]^\phi & P \\
\ar@{-}[rrrr] & & & & \\
& UP \ar[r]^\psi & UF\Phi \ar[r]^{U\phi} & UP & =  & UP \ar[r]^1 & UP \\
& & & & \ar@{-}[rrr] & & & ,\\
& & & & & FUP \ar[r]^\epsilon & P
}
\]
as required.
This induces a map
\[
\xymatrix{
	FUP \booar[r] \ar@/u0.7cm/[rr]^\epsilon \ar[d]_{\trans\psi}
	& \WkP \lffar[r] \unar[d]^\omega
	& P \ar[d]^1 \\
	F\Phi \booar[r] \ar@/d0.7cm/[rr]_\phi
	& Q \lffar[r]
	& P
}
\]
We will show that $\omega$ is pseudo-inverse to $\chi$.
Now,
\[
\xymatrix{
	\WkP \lffar[r] \ar[d]_\omega
	& P \ar[d]^1 \\
	Q \lffar[r] \ar[d]_\chi
	& P \ar[d]^1 \\
	\WkP \lffar[r]
	& P
}
\]
commutes.
So
\[
\fork \WkP{1_Q}{\chi\omega}\WkP{}P
\] is a fork.
By Lemma \ref{lem:fork}, $\chi\omega \cong 1_{{\rm Wk}(P)}$, and
similarly $\omega\chi \cong 1_Q$.
So $Q \simeq \WkP$ as a symmetric \Cat-operad, as required.
\end{proof}
\begin{corollary}
Let $Q$ be a categorification of $P$ with respect to some signature $(\Phi,
\phi)$.
Then {\rm $\Algwk(Q) \simeq \Algwk(\WkP)$}.
\end{corollary}
\begin{proof}
By a straightforward extension of the proof in \cite{hohc} Theorem 3.2.3,
$\Algwk$ is a 2-functor $\CatSymmOperad \to \CAT^{\textrm{co op}}$, and thus
preserves equivalences.
\end{proof}

\end{section}
\begin{section}{Symmetric monoidal categories}
\label{sec:symmmoncats}
\def\Csmc{\ensuremath{(\C, \otimes, I, \alpha, \lambda, \rho, \tau)}}
\def\Csmcprime{\ensuremath{(\C, \otimes', I', \alpha', \lambda', \rho', \tau')}}
Let $Q$ be the categorification of the terminal symmetric operad $P$ (whose
algebras are commutative monoids) with respect to the standard signature $\Phi$
given in Example \ref{ex:commsig} above - recall that $\Phi$ has a nullary
element $e$ and a binary element $.$\ , and is empty in all other arities.
Here we prove our assertion in Example \ref{ex:smc} that the algebras for $Q$
are classical symmetric monoidal categories.
More precisely, we show that for a given category \C, the $Q$-algebra
structures on \C\ are in one-to-one correspondence with the \smc\ structures
on \C.

Recall that a \smc\ is a structure \Csmc\, where
\begin{itemize}
\item \C\ is a category,
\item $\otimes: \C \times \C \to \C$,
\item $I \in \C$,
\item $\alpha: -\otimes(-\otimes-) \to (-\otimes-)\otimes-$,
\item $\lambda: I\otimes- \to 1_\C$,
\item $\rho: -\otimes I \to 1_\C$,
\item $\tau: (12)\cdot \otimes \to \otimes$, where $(12)$ is the non-identity
permutation of $\{1,2\}$,
\end{itemize}
satisfying the axioms given in \cite{catwork} Chapter XI.

Let \Csmc\ be a \smc.
We will define a $Q$-algebra structure  on \C, which we shall call $S\Csmc =
(\C, (\hatmap))$.
Since there's a bijective-on-objects map of symmetric \Cat-operads from $F\Phi$
to $Q$, we may define the action of the objects of $Q$ on \C\ by giving a
\Cat-operad map $(\hatmap'): F\Phi \to \End(\C)$.
Equivalently, we may give a map $\Phi \to U \End(\C)$ in $\Set^\natural$, which amounts to a choice of a functor $\C \times \C \to \C$ and an element of \C.
Let these be $\otimes$ and $I$.
To define the actions of the morphisms of $Q$ on \C, we need a natural
transformation $\hat \delta :\hat q_1 \to \hat q_2$ for each arrow $\delta:
q_1 \to q_2 \in Q_n$.
By construction of $Q$, there is such an arrow whenever $q_1$ and $q_2$,
considered as derived operations of the theory of commutative monoids, evaluate
to the same operation.
By standard properties of commutative monoids, this means that we want a
natural transformation $\hat q_1 \to \hat q_2$ iff $q_1$ and
$q_2$ take the same number of arguments.
The coherence theorem for classical symmetric monoidal categories
(\cite{catwork} XI.1) gives us exactly this (via the ``canonical'' maps), and
ensures that the composite and tensor of two such canonical maps are canonical,
i.e. that we have a well-defined map of \Cat-operads $Q \to \End(\C)$.
Hence, $(\C, (\hatmap))$ is a well-defined $Q$-algebra.

Now, let \C\ be a $Q$-algebra, with map $(\hatmap): Q \to \End(\C)$.
We shall construct a \smc\ $R(\C, (\hatmap)) = \Csmc$.
Take
 \begin{itemize}
\item $\otimes = (\hat .)$
\item $ I = \hat e$
\item $ \alpha = \hat \delta_1$, where $\delta_1 : -.(-.-) \to (-.-).-$ in $Q$,
\item $ \lambda = \hat \delta_2$, where $\delta_2 : e.- \to -$,
\item $ \rho = \hat \delta_3$, where $\delta_3 : -.e \to e $,
\item $ \tau = \hat \delta_4 $, where $\delta_4 : (12)\cdot(-.-) \to (-.-)$.
\end{itemize}
Each $\delta_i$ is uniquely defined by its source and target, since each $Q_n$
is a poset.
Because there's at most one map $q_1 \to q_2$ for any $q_1, q_2
\in Q_n$, all diagrams involving these commute.
In particular, the axioms for a \smc\ are satisfied.
So $(\C, \otimes, I, \alpha, \lambda, \rho, \tau)$ is a \smc.

Now, let \Csmc\ be a \smc.
We wish to show that $RS\Csmc = \Csmc$.
Let $RS\Csmc = \Csmcprime$.
Their underlying categories are equal, both being $\C$.
Furthermore,
\[
\begin{array}{rcccl}
\otimes' &= &\hat . &= &\otimes \\
I' &= & \hat e &= & I \\
\alpha' & = &\hat \delta_1 & = & \alpha \mbox{, the unique canonical map of
the correct type}\\
\lambda' & = &\hat \delta_2 & = & \lambda \\
\rho' & = &\hat \delta_3 & = & \rho \\
\tau' & = &\hat \delta_4 & = & \tau
\end{array}
\]

Conversely, let $(\C, (\hatmap))$ be a $Q$-algebra, and let $SR(\C,
(\hatmap)) = (\C', (\hatmap'))$.
Does $(\C, (\hatmap)) = (\C', (\hatmap'))$?
Their underlying categories are the same.
As above, $(\hatmap')$ is determined on objects by the values it takes on $.$
and $e$: these are $\otimes = \hat .$ and $I = \hat e$ respectively.
So $(\hatmap') = (\hatmap)$ on objects.
If $\delta : \tau_1 \to \tau_2$, then $\hat \delta'$ is the unique canonical
map from $\hat \tau_1' \to \hat \tau_2'$, which, by an easy induction, must be
$\hat \delta$.
So $(\hatmap') = (\hatmap)$, and hence $(\C, (\hatmap)) = SR(\C, (\hatmap))$.

\end{section}
\begin{section}{Relation to pseudo-algebras}
\label{sec:psalg}
Another approach to categorification of theories might be to promote the induced
monad $T$ on \Set\ to a 2-monad on \Cat\ via the discrete category functor,
and then to take pseudo-algebras for this 2-monad.
It is well-known to the 2-categorical cognoscenti that this is equivalent to
our construction for strongly-regular theories.
However, for linear theories, this is not the case: for instance, the
pseudo-algebras for the free commutative monoid monad are the strictly
symmetric monoidal categories, i.e. those symmetric monoidal categories
for which the symmetry map $\tau_{AB}$ is an identity for all $A,B$.
\end{section}

\begin{section}{General algebraic theories}
\label{sec:nonlinear}
In light of Lemma \ref{lem:monad}, one might ask the following question.
It is easy to extend the definition of operad (as was done, for instance, in
\cite{tronin}) so that every finitary algebraic theory is represented by one of these more general operads.
The definition of categorification presented in this paper extends
straightforwardly to this situation, to give a categorification of \emph{any}
algebraic theory.
Is this definition sensible?

Unfortunately, it isn't.
The two classes of arrows in our factorization system are, again, the levelwise
bijective-on-objects arrows and the levelwise full-and-faithful arrows; this
corresponds to a categorified theory in which \emph{all} diagrams commute, but
this is not what we want.
For instance, in the theory of commutative monoids, this would imply that the
diagram
\[
\xymatrix{
	A\otimes A \ar@<0.5ex>[r]^1 \ar@<-0.5ex>[r]_{\tau_{A,A}}
	& A\otimes A
}
\]
commutes, where $\tau$ is the symmetry map.
This is not the case for most interesting symmetric monoidal categories.
Moreover, most symmetric monoidal categories are not even equivalent to one
with this property.
\end{section}

\bibliographystyle{plain}
\bibliography{presentation}

\end{document}